\documentclass[a4paper,reqno,11pt,oneside]{amsart}
\usepackage[left=2.7cm,right=2.7cm,top=3.5cm,
bottom=3.5cm]{geometry}
\usepackage[colorlinks=true,urlcolor=blue,
citecolor=red,linkcolor=blue,linktocpage,pdfpagelabels,
bookmarksnumbered,bookmarksopen]{hyperref}
\usepackage[english]{babel}
\usepackage{graphicx}
\usepackage[toc,page,header]{appendix}
\usepackage{mathrsfs}
\usepackage{amssymb, amsmath, amsfonts, amsthm, mathtools}
\usepackage{xcolor}
\usepackage[shortlabels]{enumitem}

\usepackage{latexsym}

\allowdisplaybreaks

\makeatletter
\@namedef{subjclassname@2020}{%
  \textup{2020} Mathematics Subject Classification}
\makeatother

\DeclareMathOperator\lsd{deg}

\DeclareMathOperator\Range{Range}

\def\cyl{\Sigma_\omega}

\def\oo{\Omega_\omega}
\def\oto{\Omega_{t \omega}}
\def\uo{u_\omega}
\def\uto{u_{t\omega}}
\def\plane{\mathbb R^{N - 1}}
\def\sobsp{H_0^1(\oo \cup \Gamma)}
\def\xnspace{H_{0, x_N}^1(\oo \cup \Gamma)}
\def\holdsp{\mathcal X}
\def\mxn{m^{x_N}}
\def\sol{\mathcal S}
\def\curve{\mathcal C}
\def\concomp{\mathcal K}

\allowdisplaybreaks[1]

\numberwithin{equation}{section}

\theoremstyle{plain}
\newtheorem{theorem}{Theorem}[section]
\theoremstyle{plain}
\newtheorem{prop}[theorem]{Proposition}
\theoremstyle{plain}
\newtheorem{lemma}[theorem]{Lemma}
\theoremstyle{plain}
\newtheorem{cor}[theorem]{Corollary}
\theoremstyle{definition}

\theoremstyle{definition}
\newtheorem{definition}[theorem]{Definition}
\theoremstyle{definition}
\newtheorem{remark}[theorem]{Remark}
\theoremstyle{definition}

\theoremstyle{plain}

\begin{document}

\renewcommand{\labelenumi}{\textit{(\roman{enumi})}}

\makeatletter
\newcommand{\proofpart}[2]{%
  \par
  \addvspace{\medskipamount}%
  \noindent\emph{Part #1: #2}\par\nobreak
  \addvspace{\smallskipamount}%
  \@afterheading
}
\makeatother

\makeatletter
\def\author@andify{%
  \nxandlist {\unskip ,\penalty-1 \space\ignorespaces}%
    {\unskip {} \@@and~}%
    {\unskip , \penalty-2}%
}
\makeatother

\title[D. G. Afonso - Semilinear equations in bounded cylinders]{Semilinear equations in bounded cylinders: \\
Morse index and bifurcation from one-dimensional solutions}

\author[Danilo Gregorin Afonso]{Danilo Gregorin Afonso}

\address[Danilo Gregorin Afonso]{Dipartimento di Scienze Pure e Applicate \\
Università degli Studi di Urbino Carlo Bo \\ Piazza della Repubblica 13, 61029 Urbino, Italy}

\email{danilo.gregorinafonso@uniurb.it}



\begin{abstract}
In this paper, we study semilinear elliptic equations in domains where there is a natural class of solutions, which depend only on one variable, and whose simple geometry reflects the geometry of the domain. We prove that under quite general assumptions, other types of solutions also exist. More precisely, we consider one-dimensional solutions in bounded cylinders and, combining a suitable separation of variables with the theory of ordinary differential equations, we show how to compute the Morse index of such solutions. The Morse index is then used to prove local and global bifurcation results.
\end{abstract}

\maketitle

\section{Introduction}
\label{sec:intro}

An old and natural question in the qualitative theory of elliptic partial differential equations is to determine the conditions under which the solutions of differential problems inherit the geometric properties of the domain.

Concerning the radial symmetry, a fundamental step was the celebrated result of Gidas, Ni, and Nirenberg (\cite{GidasNiNirenberg1979}), which is based on the moving planes method and roughly states that positive solutions to semilinear Dirichlet problems in balls must be radial and radially decreasing. See also \cite{GidasNiNirenberg1981, BerestyckiNirenberg1988, BerestyckiNirenberg1990, BerestyckiNirenberg1991} for related results and developments of the method. Regarding the uniqueness/multiplicity of radial solutions, see \cite{NiNussbaum1985}. On the other hand, there is also a vast literature on symmetry-breaking results. See, for example, the papers \cite{GladialiGrossiPacellaSrikanth2010, Gladiali2010, Gladiali2017, GladialiIanni2020, CiraoloPacellaPolvara2023}, the monograph \cite{DamascelliPacella2019}, and the references therein.

Radial solutions can be seen as depending only on one variable, the distance to the central point. Our aim in this paper is to analyze another kind of ``one-dimensional solution", which is not related to invariance with respect to the action of any symmetry group but is nonetheless quite natural in the geometric setting we consider.

Let us denote a point in $\mathbb R^N$ by $x = (x', x_N)$, where $x' \in \plane$ and $x_N \in \mathbb R$, and let $\omega \subset \plane$, $N \geq 2$, be a smooth bounded domain. We consider in $\mathbb R^N$ the infinite half-cylinder spanned by $\omega$,
\begin{equation*}
\cyl \coloneqq \omega \times (0, + \infty) = \{x = (x', x_N) \in \mathbb R^N \ : \ x' \in \omega, \ x_N > 0\},
\end{equation*}
and set
\begin{equation*}
\oo \coloneqq \{x = (x', x_N) \in \cyl \ : \ x_N < 1\}
\end{equation*}
and
\begin{equation*}
\Gamma_0 \coloneqq \{x = (x', x_N) \in \cyl \ : \ x_N = 1\}.
\end{equation*}
Finally, we denote $\Gamma \coloneqq \partial \oo \setminus \overline \Gamma_0$.

The set $\Gamma_0$ is usually said to be the relative boundary of the bounded cylinder $\oo$ with respect to $\cyl$. We are interested in studying relative Dirichlet problems (i.e., with the Dirichlet condition imposed only on $\Gamma_0$) for semilinear elliptic equations in $\oo$. More precisely, for $f \in C^{1, \alpha}(\mathbb R)$, $\alpha \in (0, 1)$, we consider the problem
\begin{equation}
\label{eq:pde}
\left\{
\begin{array}{rcll}
- \Delta u & = & f(u) & \quad \text{ in } \oo \\
u & = & 0 & \quad \text{ on } \Gamma_0 \\
\displaystyle \frac{\partial u}{\partial \nu} & = & 0 & \quad \text{ on } \Gamma
\end{array}
\right.
,
\end{equation}
where $\nu$ denotes the outer unit normal vector to $\partial \oo$. The Neumann condition is the natural boundary condition to impose on the boundary of $\cyl$, for example, if one thinks that the PDE models the diffusion of some quantity $u$ which should not leave $\cyl$. Furthermore, under appropriate hypotheses on the nonlinearity $f$, this problem is analogous to pure Dirichlet problems in terms of the variational formulation. 

We say that $u \in \sobsp$ is a weak solution of \eqref{eq:pde} if
\begin{equation*}
    \int_{\oo} \nabla u \nabla v \ dx - \int_{\oo} f(u) v \ dx = 0 \quad \forall v \in \sobsp,
\end{equation*}
where $\sobsp$ is the subspace of $H^1(\oo)$ of functions whose trace vanishes on $\Gamma_0$. Whenever they exist, weak solutions are the critical points of the functional
\begin{equation}
\label{eq:def_functional}
J(v) = \frac{1}{2} \int_{\oo} |\nabla v|^2 \ dx - \int_{\oo} F(v) \ dx, \quad v \in \sobsp,
\end{equation}
where $F(s) = \int_0^s f(r) \ dr$.

Our interest in relative problems in cylinders comes from the fact that, under certain points of view, the domain $\oo$ is analogous in $\Sigma_\omega$ to balls in $\mathbb R^N$ (or, more in general, to spherical sectors in cones). One such aspect is that $\Gamma_0$ is the only surface of constant mean curvature in a certain class (see \cite{AfonsoIacopettiPacella2023Cheeger} and the references therein).  

Another analogy with balls in $\mathbb R^N$ is that in $\oo$ there are special solutions of \eqref{eq:pde} that inherit the geometry of the cylinder, in the sense that they depend solely on the ``height" variable $x_N$, and which could be analogous to radial solutions in balls. We call them one-dimensional solutions and denote them by $\uo$. They can be obtained by solving, for example by some standard variational method, the ordinary differential equation
\begin{equation*}
\left\{
\begin{array}{ll}
- u'' = f(u) & \quad \text{ in } (0, 1) \\
u'(0) = u(1) = 0 & 
\end{array}
\right.
\end{equation*}
and then setting
\begin{equation*}
\uo(x', x_N) = \widetilde u(x_N), \quad (x', x_N) \in \oo,
\end{equation*}
where $\widetilde u$ is a solution of \eqref{eq:ode}. Note that one-dimensional solutions satisfy the relative overdetermined condition $\frac{\partial \uo}{\partial \nu} = \text{constant}$ on $\Gamma_0$, and we obtain a parallel with balls once again, in view of the celebrated result of Serrin (\cite{Serrin1971}). 

A natural question is then to understand if there is an analogous of the famous result of Gidas, Ni, and Nirenberg, \cite{GidasNiNirenberg1979}, namely, if all positive solutions are one-dimensional. This is the main objective of the present paper, and we anticipate that the answer to this question is negative: there are other kinds of solutions, under very general assumptions on the data $f$ and $\omega$, so the parallel between balls in $\mathbb R^N$ and $\oo$ breaks down. 

If instead one considers the analogous of Serrin's problem, the parallel also breaks down: it is shown in \cite{AfonsoIacopettiPacella2024Energypublished} that, under certain conditions (both on $\omega$ and $f$), one may expect that there exist domains in the cylinder $\cyl$ whose relative boundaries are not flat and which admit solutions to the relative overdetermined problem. When $f \equiv 1$ and $\omega$ is an interval, the existence of such domains was already contained in \cite{FallMinlendWeth2017}. Recently, the existence of such domains was extended to general bases $\omega \subset \mathbb R^{N - 1}$ (but always for $f \equiv 1$), see \cite{PacellaRuizSicbaldi2024}.

One-dimensional solutions are not so much considered in the literature as their radial counterparts, and, to my knowledge, have always been studied in unbounded domains, since the pioneer works of Berestycki, Caffarelli, and Nirenberg (\cite{BerestyckiCaffarelliNirenberg1997a, BerestyckiCaffarelliNirenberg1997b}) in the half-space $\mathbb R^{N - 1} \times (0, + \infty)$, which lead to a famous conjecture later disproved by Sicbaldi in dimension $N \geq 3$ (\cite{Sicbaldi2010}, see also \cite{RosRuizSicbaldi2017, RosRuizSicbaldi2019}). Always on the classification of one-dimensional solutions in half-spaces, see \cite{Cortazaretal2016, FarinaSciunzi2017}, and also \cite{delPinoKowalczykWei2011} on a conjecture of De Giorgi.

Mixed boundary value problems in (unbounded) cylindrical domains $\omega \times (0, \infty)$ have been treated in \cite{ChenWuYao2023}. In this work, strong rigidity results are found, in the sense that solutions to the differential problems under analysis are found to be one-dimensional under quite weak assumptions. The authors also obtain conditions for periodicity/monotonicity of the one-dimensional solutions.

In our bounded framework, the situation is quite the opposite. Roughly speaking, our main result is that, for a wide class of nonlinearities $f$, there exist solutions to \eqref{eq:pde} which are not one-dimensional, for ``almost every shape" of the base $\omega$. This is achieved by means of bifurcation theory, as we show that solutions that are not one-dimensional bifurcate from one-dimensional solutions. One key aspect of our geometric setting is that the Morse index of one-dimensional solutions is easy to compute. We refer to Sections \ref{sec:computation_Morse_index_one_dimensional_solutions} and \ref{sec:bifurcation} for detailed statements and further comments.

This paper is organized as follows. Section \ref{sec:notations_assumptions} collects the main assumptions and recalls some preliminaries. In Section \ref{sec:computation_Morse_index_one_dimensional_solutions} we carry out a detailed study of one-dimensional solutions, showing how to compute the Morse index and that they are always nondegenerate in the space of one-dimensional functions. In the final Section \ref{sec:bifurcation} we prove our main results on bifurcation of other solutions from one-dimensional solutions.

\section{Notations, assumptions, and preliminaries}
\label{sec:notations_assumptions}

We assume that $f$ is a $C^{1, \alpha}$ nonlinearity such that
\begin{enumerate}[label=${(f_\arabic*)}$]
    \item \label{it:f1} the superlinear condition
    \begin{equation}
        \label{eq:superlinear_condition}
        f'(s) > \frac{f(s)}{s} \quad \forall s \in \mathbb R \setminus \{0\};
    \end{equation}

    \item \label{it:f2} a solution of \eqref{eq:pde} exists;

    \item \label{it:f3} $f$ ``preserves the sign":
    \begin{equation}
        \label{eq:sign_of_f}
        sf(s) > 0 \quad \forall s \neq 0.
    \end{equation}
\end{enumerate}

It is widely known that there is a plethora of classes of nonlinearities $f$ for which variational methods yield both positive and sign-changing solutions of \eqref{eq:pde} in $\sobsp$ (recall that $\sobsp$ is the subspace of $H^1(\oo)$ of functions whose trace vanishes on $\Gamma_0$). Of course, the same variational methods, when restricted to the space $\xnspace$ (which is the subspace of functions in $H_0^1(\oo \cup \Gamma)$ which depend only on the $x_N$ variable), yield the existence of one-dimensional solutions.

The archetype function satisfying \ref{it:f1}-\ref{it:f3} is the Lane-Emden nonlinearity
\begin{equation*} 
f(u) = |u|^{p - 2}u,
\end{equation*}
for $p \in (1, p^*)$, where $p^* = \frac{N + 2}{N - 2}$, if $N \geq 3$, or $p^* = \infty$ if $N = 2$, is the critical exponent for the Sobolev embedding $H_0^1(\oo \cup \Gamma) \hookrightarrow L^2(\oo)$. More in general, one could consider $f$ satisfying \ref{it:f1} and, in addition:
\begin{itemize}
    \item subcritical growth: there exists $a \in L^{\frac{2N}{N + 2}}(\oo)$ and $b > 0$ such that 
    \begin{equation*}
        |f(s)| < a(x) + b|s|^p \quad \forall s \in \mathbb R,
    \end{equation*}
    with $p \in (1, p^*)$;

    \item $f(s) = o(s)$ as $s \to 0$;

    \item Ambrosetti-Rabinowitz condition: there exists $\theta > 2$ and $r > 0$ such that for
    \begin{equation*}
        0 < \theta F(s) \leq s f(s) \quad \forall |s| \geq r.
    \end{equation*}
\end{itemize}
Then there exist positive solutions via either the Mountain Pass Theorem or minimization on the Nehari manifold. Moreover, it can be shown that such solutions have Morse index one (see the end of this section for a brief recall of the Morse index). We refer to \cite{DamascelliPacella2019} for the details.

For many more examples of classes of nonlinearities for which \eqref{eq:pde} has a solution, we refer the reader to \cite{Rabinowitz1986, Willem1996, BadialeSerra2011}.

\begin{remark}
\label{rem:regularity}
Since $f$ is assumed to be of class $C^{1, \alpha}$, then every weak solution of \eqref{eq:pde} is also in $C^2(\overline \oo)$. This follows by standard regularity arguments, considering the boundary conditions and the fact that $\partial \oo$ is made by the union of $(N - 1)$-dimensional smooth manifolds (with boundary) intersecting orthogonally (see \cite[Proposition 6.1]{PacellaTralli2020}).
\end{remark}

Let $\uo$ be a one-dimensional solution of \eqref{eq:pde}. Recall that they can be found by solving the ODE
\begin{equation}
\label{eq:ode}
\left\{
\begin{array}{ll}
- u'' = f(u) & \quad \text{ in } (0, 1) \\
u'(0) = u(1) = 0 & 
\end{array}
\right.
\end{equation}
and then setting
\begin{equation}
\label{eq:def_one_dim_sol}
\uo(x', x_N) = \widetilde u(x_N), \quad (x', x_N) \in \oo,
\end{equation}
where $\widetilde u$ is a solution of \eqref{eq:ode}. Notice that reflecting $\widetilde u$ we obtain a solution of the Dirichlet problem in $(-1, 1)$. At the same time, every solution of this Dirichlet problem satisfies also \eqref{eq:ode}, by the results of \cite{GidasNiNirenberg1979}, and thus the two problems are equivalent.

We will sometimes write $\uo$ also for the solution of \eqref{eq:ode} in $(0, 1)$. The meaning will be clear from the context.

The proofs of our main results are based on the theory of bifurcation, and our arguments rely heavily on the interplay between two eigenvalue problems. The first one is the eigenvalue problem for the linearized operator associated with the ODE \eqref{eq:ode}: 
\begin{equation}
\label{eq:auxiliary_one_dim_eigenvalue_problem}
\left\{
\begin{array}{ll}
- z'' - f'(\uo) z = \alpha z & \quad \text{ in } (0, 1) \\
z'(0) = z(1) = 0 & 
\end{array}
\right.
.
\end{equation}
We denote its eigenvalues, which are all simple (see \cite{Hartman2002book}), by $\alpha_i$, $i \in \mathbb N^+$. The second one is for the Neumann-Laplacian in the $(N - 1)$-dimensional domain $\omega$ that spans the cylinder:
\begin{equation}
\label{eq:Neumann_eigenvalue_problem_in_omega}
\left\{
\begin{array}{rcll}
- \Delta_{\plane} \psi & = & \lambda \psi & \quad \text{ in } \omega \\
\displaystyle \frac{\partial \psi}{\partial \nu} & = & 0 & \quad \text{ on } \partial \omega
\end{array}
\right.
,
\end{equation}
where $- \Delta_{\plane} = - \sum_{j = 1}^{N - 1} \frac{\partial^2}{\partial x_j^2}$ is the Laplace operator in $\plane$, that is, with respect to the variable $x' = (x_1, \ldots, x_{N - 1})$. We denote its eigenvalues, counted with multiplicity, by
\begin{equation*}
0 = \lambda_0(\omega) < \lambda_1(\omega) \leq \lambda_2(\omega) \leq \ldots.
\end{equation*}

The Morse index $m(u)$ of a weak solution $u$ is given by the number of negative eigenvalues of the linearized operator
\begin{equation}
\label{eq:linearized_operator}
L_u = - \Delta - f'(u)
\end{equation}
in the space $\sobsp$. Note that if $u = \uo$ is a one-dimensional solution, we can restrict our attention to the eigenvalues of $L_{\uo}$ whose corresponding eigenfunctions are one-dimensional, which are precisely the eigenvalues of \eqref{eq:auxiliary_one_dim_eigenvalue_problem}. We denote by $\mxn(\uo)$ the Morse index of $\uo$ in $\xnspace$, which is the subspace of functions in $H_0^1(\oo \cup \Gamma)$ which depend only on the $x_N$ variable. A solution $u$ is said to be degenerate if the linearized operator $L_{u}$ admits zero as an eigenvalue, or, in other words, if there exists a nontrivial weak solution $\varphi \in \sobsp$ to the problem
\begin{equation*}
\left\{
\begin{array}{rcll}
- \Delta \varphi - f'(u) \varphi & = & 0 & \quad \text{ in } \oo \\
\varphi & = & 0 & \quad \text{ on } \Gamma_0 \\
\displaystyle \frac{\partial \varphi}{\partial \nu} & = & 0 & \quad \text{ on } \Gamma
\end{array}
\right.
.
\end{equation*}
Analogous definitions and statements hold in the subspace $\xnspace$ of one-dimensional functions and also for functions of one real variable in the space 
$$
\mathcal H^1(0, 1) \coloneqq \{z \in H^1(0, 1) \ : \ z(1) = 0\}.
$$
For more details on these matters, we refer the reader to \cite{DamascelliPacella2019}.

Let $u$ be a weak solution of \eqref{eq:pde} (which is also a classical solution, in view of Remark \ref{rem:regularity}). We denote by $\mathcal Z(u)$ the nodal set of $u$ in $\oo$:
\begin{equation*}
    \mathcal Z(u) = \{x \in \oo \ : \ u(x) = 0\}.
\end{equation*}
By a nodal domain, we mean a connected component of $\oo \setminus \mathcal Z(u)$, that is, a region of $\oo$ where $u$ does not change sign. We denote by $n(u)$ the number of nodal domains of the solution $u$.

It is well-known that the superlinear condition \eqref{eq:superlinear_condition} implies that the Morse index of a solution $u$ is bounded from below by the number of nodal regions:
\begin{equation}
    \label{eq:bound_m_n}
    m(u) \geq n(u).
\end{equation}
Indeed, the Morse index can be characterized as the maximal dimension of a subspace of $\sobsp$ where the quadratic form associated with the linearized operator $L_u$ is negative definite. If $u$ satisfies \eqref{eq:pde}, then multiplying the equation by $u$ and integrating by parts yield
\begin{equation*}
    \int_{\oo} |\nabla u|^2 \ dx = \int_{\oo} f(u) u \ dx.
\end{equation*}
On the other hand, the quadratic form associated to $L_u$ is
\begin{equation*}
    L_u[v, v] = \int_{\oo} |\nabla v|^2 \ dx - \int_{\oo} f'(u) v^2 \ dx, \quad v \in \sobsp.
\end{equation*}
Hence \eqref{eq:superlinear_condition} implies 
\begin{align}
    L_u[u, u]
    & = \int_{\oo} |\nabla u|^2 \ dx - \int_{\oo} f'(u) u^2 \ dx \nonumber \\
    & = \int_{\oo} \left(\frac{f(u)}{u}  - f'(u)\right)u^2 \ dx \nonumber \\
    & < 0. \nonumber 
\end{align}
This argument can be repeated using the restriction of $u$ to a nodal region, in such a way that each nodal region yields a (linearly independent) direction where the second bilinear form corresponding to $L_u$ is negative definite, which in turn yield \eqref{eq:bound_m_n}.

For the bifurcation arguments, we will consider a fixed $(N - 1)$-dimensional smooth bounded domain $\omega$, consider scalings $t\omega$, with $t > 0$, and use $t$ as the bifurcation parameter. We denote by $\oto$ the bounded cylinder spanned by the scaled domain, and by $\Gamma_{0, t}$ and $\Gamma_t$ the parts of $\partial \oto$ corresponding to $\Gamma_0$ and $\Gamma$, respectively.

\section{One-dimensional solutions}
\label{sec:computation_Morse_index_one_dimensional_solutions}

\subsection{Decomposition of the spectrum of the linearized operator}

Let $\uo$ be a one-dimensional solution. Recall that we adopt the notation $\uo$ also for the real function $\widetilde u$ that solves \eqref{eq:ode} in $(0, 1)$. To analyze the spectrum of $L_{\uo}$, it is convenient to consider the one-dimensional eigenvalue problem \eqref{eq:auxiliary_one_dim_eigenvalue_problem}.

In what follows, we denote by $L_{\uo}^{x_N}$ the linear operator defined by
\begin{equation*}
L_{\uo}^{x_N}(z) = - z'' - f'(\uo) z, \quad z \in \mathcal H^1(0, 1).
\end{equation*}
It is immediately seen that the Morse index of $\uo$ in $\mathcal H^1(0, 1)$ is precisely $\mxn(\uo)$ in $\xnspace$. 

We have the following decomposition of the spectrum of $L_{\uo}$, obtained in \cite{AfonsoIacopettiPacella2024Energypublished} and which we report here for the convenience of the reader:

\begin{lemma}
\label{lem:spectrum_decomposition}
Let $\uo$ be a one-dimensional solution of \eqref{eq:pde}. The spectra of $L_{\uo}$, $L_{\uo}^{x_N}$ and $- \Delta_{\plane}$ are related by
\begin{equation*}
    \sigma(L_{\uo}) = \sigma(L_{\uo}^{x_N}) + \sigma(- \Delta_{\plane}),
\end{equation*}
where $\sigma(\cdot)$ denotes the spectrum of a linear operator.
\end{lemma}

\begin{proof}
We begin by showing that $\sigma(L_{\uo}) \subset \sigma(L_{\uo}^{x_N}) + \sigma(- \Delta_{\plane})$. Let $\mu \in \sigma(L_{\uo})$ and let $\varphi \in \sobsp$ be an associated eigenfunction, that is, $\varphi$ is a weak solution of
\begin{equation*}
\left\{
\begin{array}{rcll}
- \Delta \varphi - f'(\uo) \varphi & = & \mu \varphi & \quad \text{ in } \oo \\
\varphi & = & 0 & \quad \text{ on } \Gamma_0 \\
\displaystyle \frac{\partial \varphi}{\partial \nu} & = & 0 & \quad \text{ on } \Gamma
\end{array}
\right.
.
\end{equation*}
As observed in Remark \ref{rem:regularity}, by the shape of $\oo$ and the boundary conditions, elliptic regularity theory yields that $\varphi$ is a classical solution in $\overline \oo$.

Let $\lambda$ be an eigenvalue of $- \Delta_{\plane}$ in $\omega$ with homogeneous Neumann boundary condition and let $\psi$ be a corresponding eigenfunction. Define
\begin{equation*}
    z(x_N) = \int_\omega \varphi(x', x_N) \psi(x') \ dx'.
\end{equation*}
Differentiating twice with respect to $x_N$, using Green's formula and the boundary conditions we obtain
\begin{align}
-z''
& = \int_\omega - \frac{\partial^2 \varphi}{\partial x_N^2} \psi \ dx' \nonumber \\
& = \int_\omega (- \Delta \varphi + \Delta_{\plane} \varphi) \psi \ dx' \nonumber \\
& = \int_\omega f'(\uo) \varphi \psi \ dx' + \int_\omega \lambda \varphi \psi \ dx' + \int_\omega \Delta_{\plane} \psi \varphi \ dx' \nonumber \\
& = f'(\uo) z + \mu z - \lambda \int_\omega \psi \varphi \ dx' \nonumber \\
& = f'(\uo) z + \mu z - \lambda z. \nonumber
\end{align}
Thus $(\mu - \lambda) \in \sigma(L_{\uo}^{x_N})$ and hence $\mu = (\mu - \lambda) + \lambda \in \sigma(L_{\uo}^{x_N}) + \sigma(- \Delta_{\plane})$.

To show the reverse inclusion, let $\alpha \in \sigma(L_{\uo}^{x_N})$, $\lambda \in \sigma(- \Delta_{\plane})$ and let $z$, $\psi$ be respective associated eigenfunctions. Setting
\begin{equation*}
\varphi(x', x_N) = z(x_N) \psi(x'), \quad (x', x_N) \in \oo
\end{equation*}
we have
\begin{align}
- \Delta \varphi
& = -z'' \psi - \Delta_{\plane} \psi z \nonumber \\
& = f'(\uo) z \psi + \alpha z \psi + \lambda z \psi \nonumber \\
& = f'(\uo) \varphi + (\alpha + \lambda) \varphi. \nonumber
\end{align}
As a consequence, $\alpha + \lambda \in \sigma(L_{\uo})$, and the proof is complete.
\end{proof}

\begin{cor}
\label{cor:condition_for_nondegeneracy}
Let $\uo$ be a one-dimensional solution of \eqref{eq:pde}. Then $\uo$ is degenerate if and only if there exist $i \in \mathbb N^+$ and $j \in \mathbb N$ such that
\begin{equation*}
\alpha_i + \lambda_j = 0.
\end{equation*}
\end{cor}

\subsection{Morse index computation}

Our objective in this section is to understand the one-dimensional Morse index $\mxn(\uo)$. To this aim, we analyze the eigenvalue problem corresponding to the linearized operator associated with the ODE \eqref{eq:ode}:
\begin{equation*}
\left\{
\begin{array}{ll}
- z'' - f'(\uo) z = \alpha z & \quad \text{ in } (0, 1) \\
z'(0) = z(1) = 0 & 
\end{array}
\right.
.
\end{equation*}
The spectral theory for this kind of problem is well-established in the literature (see, e.g., \cite[Chapter XI]{Hartman2002book}). It is known that there exists a sequence of real eigenvalues 
$$
\alpha_1 < \alpha_2 < \ldots < \alpha_i < \ldots,
$$
with $\alpha_i \to + \infty$ as $i \to \infty$; corresponding to each eigenvalue $\alpha_i$, there exists a unique (up to multiplicative constants) eigenfunction $z_i$; for each $i \in \mathbb N^+$, the eigenfunction $z_i$ has precisely $i - 1$ zeros in $(0, 1)$.

Our first result relates the one-dimensional Morse index $\mxn$ with the number of nodal regions.

\begin{prop}
\label{prop:computation_of_mxn}
Assume that \ref{it:f1}-\ref{it:f3} hold. Let $\uo \in \sobsp$ be a one-dimensional solution of \eqref{eq:pde} with $n(\uo) = n \geq 1$ nodal domains. Then
\begin{equation*}
    \mxn(\uo) = n.
\end{equation*}
\end{prop}

\begin{proof}
    It is well-known (see Section \ref{sec:notations_assumptions}) that $\mxn(\uo) \geq n$. Since $\mxn(\uo)$ is the number of negative eigenvalues of \eqref{eq:auxiliary_one_dim_eigenvalue_problem}, we have that
    \begin{equation*}
        \alpha_1 < \alpha_2 < \ldots < \alpha_n < 0.
    \end{equation*}
    Our aim is then to show that $\alpha_{n + 1} \geq 0$.

    Since $\uo$ satisfies the equation $- \uo'' = f(\uo)$ in the interval $(0, 1)$, taking the derivative with respect to $x_N$ and rearranging terms we obtain
    \begin{equation*}
        (\uo')'' + f'(\uo) \uo' = 0 \quad \text{ in } (0, 1)
    \end{equation*}
    Moreover, since $\uo$ has $n$ nodal regions, then we known that $\uo'$ has $n - 1$ zeros in $(0, 1)$ (see Remark \ref{rem:roots_of_u_prime}).

    On the other hand, from the spectral theory for \eqref{eq:auxiliary_one_dim_eigenvalue_problem}, we known that the eigenfunction $z_{n + 1}$ has $n$ zeros in $(0, 1)$ (and another zero at $x_N = 1$) and satisfies
    \begin{equation*}
        z_{n + 1}'' + (f'(\uo) + \alpha_{n + 1}) z_{n + 1} = 0 \quad \text{ in } (0, 1)
    \end{equation*}

    Suppose, for the sake of contradiction, that $\alpha_{n + 1} < 0$. Should this be the case, we would have $f'(\uo) + \alpha_{n + 1} < f'(\uo)$ and therefore the Sturm-Picone comparison theorem would imply that $\uo'$ has at least one zero between each consecutive zero of $z_{n + 1}$. Since $z_{n + 1}$ has $n + 1$ zeros in $[0, 1]$, then it would follow that $\uo'$ has at least $n$ zeros in $(0, 1)$, which is a contradiction with the assumption on the number of nodal regions of $\uo$. Then $\alpha_{n + 1}$ cannot be negative, and the proof is complete.
\end{proof}

\begin{remark}
    \label{rem:roots_of_u_prime}
    In the proof of Proposition \ref{prop:computation_of_mxn} we have tacitly used that between each zero of $\uo$ in $[0, 1]$ there is one and only one critical point (that is to say that $\uo'$ has one and only one zero). Indeed, this follows by arguing as in the classical theorem of Gidas, Ni, and Nirenberg (\cite{GidasNiNirenberg1979}) in each nodal domain of $\uo$, where $\uo$ satisfies a Dirichlet boundary value problem.
\end{remark}

Making use of the decomposition of the spectrum of $L_{\uo}$ obtained in Lemma \ref{lem:spectrum_decomposition}, we can easily compute the Morse index of a one-dimensional solution in the space $\sobsp$:

\begin{theorem}
\label{thm:computation_of_Morse_index}
Assume that \ref{it:f1}-\ref{it:f3} hold. Let $\uo$ be a one-dimensional solution of \eqref{eq:pde} in $\oo$. Then it holds:
\begin{enumerate}

\item if $\mxn(\uo) = 1$, then
\begin{equation*}
m(\uo) = 1 + \# \{j \geq 1 \ : \ \lambda_j < -\alpha_1 \}
\end{equation*}

\item if $\mxn(\uo) = n$, then
\begin{equation*}
m(\uo) = n + \sum_{i = 1}^n \# \{j \geq 1 \ : \ \lambda_j < - \alpha_i\}.
\end{equation*}
\end{enumerate}
\end{theorem}

\begin{proof}
For \textit{(i)}, we notice that, since $\mxn(\uo) = 1$, by definition we have
$$
\alpha_1 < 0 \leq \alpha_i \quad \forall i \geq 2.
$$
Therefore, since $\lambda_0 = 0$ and $\lambda_j > 0$ for $j \geq 1$, we have
\begin{align}
m(\uo)
& = \# \{k \geq 1 \ : \ \lambda_k < 0\} \nonumber \\
& = \# \{j \geq 0 \ : \ \alpha_1 + \lambda_j < 0\} \nonumber \\
& = 1 + \# \{j \geq 1 \ : \ \lambda_j < - \alpha_1\}. \nonumber 
\end{align}

For \textit{(ii)}, since $\lambda_0 = 0$ we have
\begin{align}
m(\uo)
& = \# \{k \geq 1 \ : \ \lambda_k < 0\} \nonumber \\
& = \# \{i \geq 1, j \geq 0 \ : \ \lambda_j < - \alpha_i\} \nonumber \\
& = \sum_{i = 1}^n \# \{j \geq 0 \ : \ \lambda_j < - \alpha_i\} \nonumber \\
& = n + \sum_{i = 1}^n \# \{j \geq 1 \ : \ \lambda_j < - \alpha_i\}, \nonumber
\end{align}
which completes the proof.
\end{proof}

Recall that scaling the domain $\omega$ by a factor $t > 0$, we scale the Neumann eigenvalues of $- \Delta_{\plane}$ in $\omega$:
\begin{equation}
\label{eq:scaled_Neumann_eigenvalue}
\lambda_j(t\omega) = \frac{1}{t^2} \lambda_j(\omega),
\end{equation}
for every $j \in \mathbb N$. Combining this with Lemma \ref{lem:spectrum_decomposition}, Corollary \ref{cor:condition_for_nondegeneracy} and Theorem \ref{thm:computation_of_Morse_index} we have the following:

\begin{cor}
Let $\omega \subset \plane$ be a smooth bounded domain and let $\uo$ be a one-dimensional solution of \eqref{eq:pde} in $\oo$. For $t > 0$, let $\uto$ be the corresponding one-dimensional solution in the scaled domain $\oto$. Then there exists a sequence $0 < \bar t_1 < \bar t_2 < \ldots$ such that $u_{\bar t_n \omega}$ is degenerate in the space $H_0^1(\Omega_{\bar t_n \omega} \cup \Gamma_{\bar t_n})$ for all $n \in \mathbb N^+$ and 
\begin{equation*}
m(\uto) \to \infty \quad \text{ as } \quad t \to + \infty.
\end{equation*} 
\end{cor}

\begin{proof}
It follows immediately from \eqref{eq:scaled_Neumann_eigenvalue} and Theorem \ref{thm:computation_of_Morse_index}.
\end{proof}

Recall that solutions of Morse index one can be obtained, for example, by minimization on the Nehari manifold. Then, our next result tells us that least energy solutions are not one-dimensional, provided that the base $\omega$ is large enough.

\begin{theorem}
    \label{thm:ground_states}
    Assume that \ref{it:f1}-\ref{it:f3} hold. Furthermore, assume that the one-dimensional solution $\uo$ is unique and there exists a solution $\bar u$ of \eqref{eq:pde} of Morse index one in $\sobsp$. Then, if
    \begin{equation}
        \label{eq:ground_states}
        \lambda_1(\omega) < - \alpha_1,
    \end{equation}
    the solution $\bar u$ is not one-dimensional.
\end{theorem}

\begin{proof}
    By Theorem \ref{thm:computation_of_Morse_index}, it follows that $m(\uo) \geq 2$, and hence $\bar u \neq \uo$. 
\end{proof}

\subsection{Nondegeneracy in the space $\xnspace$}
\begin{prop}    
    \label{prop:nondegeneracy}
    Assume that \ref{it:f1}-\ref{it:f3} hold. Let $\uo$ be a one-dimensional solution of \eqref{eq:pde}. Then $\uo$ is nondegenerate in $\xnspace$.
\end{prop}

\begin{proof}
    Let $n = n(\uo)$ be the number of nodal regions of $\uo$. By Proposition \ref{prop:computation_of_mxn}, we have that
    \begin{equation*}
        \mxn(\uo) = n,
    \end{equation*}
    and therefore $\alpha_1 < \alpha_2 < \ldots < \alpha_n < 0 \leq \alpha_{n + 1} < \ldots$. Our aim is then to show that $\alpha_{n + 1} > 0$.

    Suppose, for the sake of contradiction, that $\alpha_{n + 1} = 0$. By Sturm-Liouville theory, the corresponding eigenfunction $z_{n + 1}$, which satisfies
    \begin{equation*}
        \left\{
        \begin{array}{ll}
            - z_{n + 1}'' - f'(\uo) z_{n + 1} = 0 & \quad \text{ in } (0, 1) \\
            z_{n + 1}'(0) = z_{n + 1}(1) = 0 & 
        \end{array}
        \right.
        ,
    \end{equation*}
    has $n$ zeros in $(0, 1)$, thus $n + 1$ nodal regions.

    Let us choose the sign of $z_{n + 1}$ so that $\uo(0) z_{n + 1}(0) > 0$. Since $z_{n + 1}$ has one more nodal region than $\uo$, we have that $\uo'(1) z_{n + 1}'(1) < 0$ (strict inequality is due to Hopf's Lemma). Now, observe that
    \begin{align}
        0
        & > \uo'(1) z_{n + 1}'(1) \nonumber \\
        & = \uo'(1) z_{n + 1}'(1) - \underbrace{\uo'(0) z_{n + 1}'(0)}_{= 0} \nonumber \\
        & = \int_0^1 (\uo'z_{n + 1}')' \ dx_N \nonumber \\
        & = \int_0^1 \uo'' z_{n + 1}' \ dx_N + \int_0^1 \uo' z_{n + 1}'' \ dx_N \nonumber \\
        & = - \int_0^1 f(\uo) z_{n + 1}' \ dx_N - \int_0^1 f'(\uo) \uo' z_{n + 1} \ dx_N \nonumber \\
        & = - \int_0^1 (f(\uo) z_{n + 1})' \ dx_N \nonumber \\
        & = f(\uo(0)) z_{n + 1}(0) - f(\uo(1))z_{n + 1}(1) \nonumber \\
        & = f(\uo(0)) z_{n + 1}(0) \nonumber \\
        & > 0, \label{eq:nondegeneracy}
    \end{align}
    where the last inequality follows from our choice on the sign of $z_{n + 1}(0)$ and \eqref{eq:sign_of_f}. Since \eqref{eq:nondegeneracy} is absurd, we have that $\alpha_{n + 1} \neq 0$. The proof is complete.
\end{proof}

\section{Bifurcation results}
\label{sec:bifurcation}

In this section, we prove the existence of solutions to \eqref{eq:pde} which are not one-dimensional. We will show that such solutions bifurcate from one-dimensional ones when the domain $\omega$ which spans the cylinder is scaled. In addition to local bifurcation results, we study the structure of the bifurcating solutions, including a global bifurcation result in the spirit of Rabinowitz (\cite{Rabinowitz1971}). See Theorem \ref{thm:bifurcation}.

We fix $\omega$, consider the scalings $t\omega$, with $t > 0$, and use $t$ as the bifurcation parameter. We denote by $\oto$ the bounded cylinder spanned by the scaled domain, and by $\Gamma_t$ the part of the boundary corresponding to $\Gamma$.

Of course, when we scale the domain $\omega$, the PDE changes, since the domain $\oto$ and the functional space $H_0^1(\oto \cup \Gamma_t)$ are different. However, it is more convenient to develop the bifurcation argument in a fixed functional setting, independent of the scaling. We choose to work in the Banach space (endowed with the $C^{1, \alpha}$ norm)
\begin{equation*} 
\holdsp(\oo \cup \Gamma) \coloneqq \left\{u \in C^{1, \alpha}(\overline \oo) \ : \ u|_{\Gamma_0} = 0, \ \left.\frac{\partial u}{\partial \nu}\right|_{\Gamma} = 0 \right\},
\end{equation*}
and we ``bring back" to this space the PDE problems considered in the scaled domains $\oto$ via suitable diffeomorphisms.

To be more precise, denote by $x = (x', x_N)$ the points in the original cylinder $\oo$, and by $y = (y', x_N)$ the points in the scaled cylinder $\oto$. Clearly, $y' = tx'$. We consider the diffeomorphism $h_t: \overline{\oto} \to \overline{\oo}$ given by
\begin{equation*}
    h_t(y', x_N) = \left(\frac{y'}{t}, x_N \right).
\end{equation*}
Naturally, functions defined in the scaled domain correspond to functions in $\oo$, via a scaling of the $x'$ argument. More precisely, the diffeomorphism $h_t$ induces the linear isomorphism $h_t^* : \holdsp(\oto \cup \Gamma_t) \to \holdsp(\oo \cup \Gamma)$ given by
\begin{equation*}
    h_t^*(v) (x', x_N) = v(h_t^{-1}(x', x_N)) = v(tx', x_N).
\end{equation*}

Let us now describe how to bring back the PDE from $\oto$ to $\oo$. Let $u_t \in \holdsp(\oto \cup \Gamma_t)$ be any solution of \eqref{eq:pde} in $\oto$. Then, the function
\begin{equation*}
    u_t^* \coloneqq h_t^*(u_t) = u_t \circ h_t^{-1} \in \holdsp(\oo \cup \Gamma)
\end{equation*}
satisfies
\begin{equation}
    \label{eq:transported_equation}
    \left\{
    \begin{array}{rcll}
        D_t(u_t^*) & = & f(u_t^*) & \quad \text{ in } \oo \\
        u_t^* & = & 0 & \quad \text{ on } \Gamma_0 \\
        \displaystyle \frac{\partial u_t^*}{\partial \nu} & = & 0 & \quad \text{ on } \Gamma
    \end{array}
    \right.
    ,
\end{equation}
where the operator $D_t$ is defined as 
\begin{equation*}
    D_t(u) = h_t^*\left(-\Delta\left((h_t^*)^{-1}(u)\right)\right).
\end{equation*}
Indeed, note that $u_t = u_t^* \circ h_t = (h_t^*)^{-1}(u_t^*)$, and thus, since $- \Delta u_t = f(u_t)$ in $\oto$, we have
\begin{equation*}
    - \Delta \left((h_t^*)^{-1}(u_t^*) \right)(y) = f\left((h_t^*)^{-1}(u_t^*)\right)(y) \quad \text{ in } \oto.
\end{equation*}
Then, writing $y = h_t^{-1}(x)$, for $x \in \overline\oo$, $D_t(u_t^*) = f(u_t^*)$ immediatelly follows. Moreover, the boundary conditions in \eqref{eq:transported_equation} are clearly satisfied.

We can further rewrite the equation $D_t(u) = f(u)$ as $u = T_t(u)$, where
\begin{equation*}
    T_t : \holdsp(\oo \cup \Gamma) \to \holdsp(\oo \cup \Gamma)
\end{equation*}
is the operator defined as
\begin{equation*}
    T_t(u) = D_t^{-1}(f(u)), \quad u \in \holdsp(\oo \cup \Gamma),
\end{equation*}
which can also be written as
\begin{equation}
    \label{eq:def_T_t}
    T_t(u) = h_t^*\left((- \Delta)^{-1}\left((h_t^*)^{-1}(f(u))\right)\right).
\end{equation}
Due to classical regularity results, this operator is well-defined and compact, and therefore, for any open bounded set $U \subset \holdsp(\oo \cup \Gamma)$ such that $I - T_t \neq 0$ on $\partial U$, the Leray-Schauder degree $\deg(I - T_t, U, 0)$ is well-defined.

Finally, note that if $\uo$ is a one-dimensional solution in $\oo$, then 
\begin{equation*}
(\uto)^* = h_t^*(\uto) = \uo.    
\end{equation*}

We can now define what we mean by bifurcating solutions in the scaled domains:
\begin{definition}
    \label{def:bifurcation}
    Let $\uo$ be a one-dimensional solution of \eqref{eq:pde} in $\oo$, and let $\bar t > 0$. We say that a bifurcation occurs at $(\bar t, u_{\bar t \omega})$ if in any neighborhood of $(\bar t, \uo)$ in $(0, + \infty) \times \holdsp(\oo \cup \Gamma)$ there exists a solution of \eqref{eq:transported_equation}. Such other solutions will be referred to as bifurcating solutions.
\end{definition}

\begin{remark}
    \label{rem:symmetry_breaking}
    Solutions of \eqref{eq:transported_equation} in $\oo$ give rise to solutions of the original PDE in $\oto$ via the action of $(h_t^*)^{-1}$. Then Definition \ref{def:bifurcation} says that a bifurcation happens if, for $t$ arbitrarily close to $\bar t$, there exists a solution of \eqref{eq:pde} in $\oto$ close to $\uto$. Moreover, since one-dimensional solutions are never degenerate in the space of one-dimensional functions (Proposition \ref{prop:nondegeneracy}), such new solutions cannot be one-dimensional.
\end{remark}

So far in this section, we have seen how to transport the PDE from the dilated domain $\oto$ back to $\oo$, to obtain an appropriate functional framework for the bifurcation. Our arguments will rely heavily on the behavior of the Leray-Schauder degree of $I - T_t$ in the space $\holdsp(\oo \cup \Gamma)$ when $t$ changes. The key idea is that due to our next result, Lemma \ref{lem:deg_T_deg_P}, we can detect changes on $\lsd(I - T_t, U, 0)$ due to the change in the Morse index of $\uto$, which in turn is a phenomenon happening in the dilated domains/respective functional spaces. 

Lemma \ref{lem:deg_T_deg_P} allows us to ``transport" the calculation of the degree to the dilating domains. Its proof is an application of the following result of Nussbaum:
\begin{prop}[{\cite[Section D, Proposition 2]{Nussbaum1971}}]
    \label{prop:Nussbaum}
    Let $V$ be an open subset of a Banach space $E_1$. Let $C: E_1 \to E_2$ be a compact map into a Banach space $E_2$. Let $\beta: \overline V \to E_2$ be a homeomorphism. Assume that $\beta(x) - C(x) \neq 0$ for $x \in \partial V$. Then
    \begin{equation}
        \label{eq:Nussbaum}
        \lsd(I -(C \circ \beta^{-1}), \beta(V), 0) = \lsd(I - (\beta^{-1} \circ C), C^{-1}(\beta(V)), 0).
    \end{equation}
\end{prop}
\begin{remark}
    It is to be understood as part of the statement that the Leray-Schauder degrees in \eqref{eq:Nussbaum} are well-defined.
\end{remark}

\begin{lemma}
    \label{lem:deg_T_deg_P}
    Let $t > 0$. Let $P_t:\holdsp(\oto \cup \Gamma_t) \to \holdsp (\oto \cup \Gamma_t)$ be the operator given by
    \begin{equation}
    \label{eq:def_P_t}
     P_t(v) = (- \Delta)^{-1}(f(v)), \quad v \in \holdsp(\oto \cup \Gamma_t).
    \end{equation}
    Let $U \subset \holdsp(\oo \cup \Gamma)$ be an open bounded set such that $I - T_t \neq 0$ on $\partial U$. Then
    \begin{equation}
        \label{eq:deg_T_deg_P}
        \lsd(I - T_t, U, 0) = \lsd(I - P_t, (h_t^*)^{-1}(U), 0).
    \end{equation}
\end{lemma}

\begin{proof}
    That the operator $P_t$ is compact is standard. Let us now consider the operator $C_t: \holdsp(\oto \cup \Gamma_t) \to \holdsp(\oo \cup \Gamma)$ given by
    \begin{equation*}
        C_t(v) = h_t^*(P_t(v)) = h_t^* \left( (-\Delta)^{-1}(f(v)) \right), \quad v \in \holdsp(\oto \cup \Gamma_t).
    \end{equation*}
    It is standard to show that $C_t$ is compact as well. Moreover, $P_t = (h_t^*)^{-1} \circ C_t$. Furthermore, recalling that
    \begin{equation*}
        T_t(u) = h_t^*\left( (-\Delta)^{-1} \left( (h_t^*)^{-1} (f(u)) \right) \right), \quad u \in \holdsp(\oo \cup \Gamma),
    \end{equation*}
    we see that $T_t = C_t \circ (h_t^*)^{-1}$. 

    Invoking Proposition \ref{prop:Nussbaum}, with $C = C_t$, $\beta = h_t^*$ and $V = (h_t^*)^{-1}(U)$, we have
    \begin{align}
        \lsd(I - T_t, U, 0) 
        & = \lsd\left(I - P_t, C_t^{-1}\left((h_t^*)\left((h_t^*)^{-1}(U)\right)\right), 0\right) \nonumber \\
        & = \lsd(I - P_t, C_t^{-1}(U), 0) \nonumber \\
        & = \lsd(I - P_t, (h_t^*)^{-1}(U), 0), \nonumber
    \end{align}
    since $C_t^{-1}(U) \subset (h_t^*)^{-1}(U)$ and $(I - P_t)(v) \neq 0$ for all $v \in (h_t^*)^{-1}(U) \setminus C_t^{-1}(U)$. Indeed, since $C_t(V) = T_t(U)$, we have
    \begin{align}
    (C_t(V) \cap U) = (T_t(U) \cap U) \subset T_t(U), \nonumber
    \end{align}
    which in turn implies that
    \begin{equation*}
        C_t^{-1}(U) \subset (h_t^*)^{-1} \left(T_t^{-1}(T_t(U)) \right) = (h_t^*)^{-1}(U),
    \end{equation*}
    and $v = P_t(v)$ if, and only if, $h_t^*(v) = C_t(v)$.
\end{proof}

\begin{theorem}
    \label{thm:bifurcation}
    Let $\omega \subset \mathbb R^{N - 1}$ be a smooth bounded domain and let $f \in C^{1, \alpha}(\mathbb R)$ be a nonlinearity satisfying \ref{it:f1}-\ref{it:f3}. Let $\uo$ be a one-dimensional solution of \eqref{eq:pde} in $\oo$ having $n(\uo) = n$ nodal regions. Set
    \begin{equation*}
        \sol \coloneqq \overline{\{(t, u) \in (0, + \infty) \times \holdsp(\oo \cup \Gamma) \ : \ (I - T_t)(u) = 0, \ u \neq \uo\}},
    \end{equation*}
    where $T_t$ is the operator defined in \eqref{eq:def_T_t}. Let $\bar t > 0$ such that $u_{\bar t\omega}$ is degenerate in the space $H_0^1(\Omega_{\bar t\omega} \cup \Gamma_{\bar t})$. Then, if $0$ is a simple eigenvalue of $L_{u_{\bar t \omega}}$ in $H_0^1(\Omega_{\bar t \omega} \cup \Gamma_{\bar t})$, it holds:
    \begin{enumerate}[label=${(\roman*)}$]
        \item \label{it:local_bifurcation} a bifurcation occurs at $(\bar t, u_{\bar t \omega})$, in the sense of Definition \ref{def:bifurcation};

        \item \label{it:continuity} for $t$ sufficiently close to $\bar t$, the branch of bifurcating solutions form a $C^1$ curve;

        \item \label{it:nodal_domains} for $t$ sufficiently close to $\bar t$, the bifurcating solutions have $n$ nodal regions;
        
        \item \label{it:global_bifurcation} the connected component of $\sol$ containing $(\bar t, \uo)$ is either unbounded in $(0, + \infty) \times \holdsp(\oo \cup \Gamma)$, or contains another point $(\widetilde t, \uo)$, with $\widetilde t \neq \bar t$ (possibly with $\widetilde t = 0$).
    \end{enumerate}
\end{theorem}

\begin{proof}

    We prove each item separately.

    \begin{enumerate}[label=${(\roman*)}$, leftmargin=*]
    \item We divide the proof of \ref{it:local_bifurcation} into two parts. First, we give the main argument leading to the local bifurcation phenomena. We assume the validity of the formula \eqref{eq:deg_P_prime}, which we prove in the second part.

    \proofpart{1}{Proof that $(\bar t, \uo)$ is a bifurcation point in $(0 + \infty) \times \holdsp(\oo \cup \Gamma)$}
    
    For the sake of contradiction, let us assume the contrary. Then we can find $\delta > 0$ and a small a neighborhood $U$ of $\uo$ in $\holdsp(\oo \cup \Gamma)$ such that $\uo$ is the unique solution of $(I - T_t)(u) = 0$ in $\overline U$ for all $t \in [\bar t - \delta, \bar t + \delta]$. In other words, $T_t$ is an admissible homotopy.  By the homotopy invariance of the Leray-Schauder degree, it holds that $\lsd(I - T_t, U, 0)$ is constant for $t \in [\bar t - \delta, \bar t + \delta]$. In particular, for any $\varepsilon \in (0, \delta)$ we have
    \begin{equation}
        \label{eq:contradiction_assumption}
        \lsd(I - T_{\bar t - \varepsilon}, U, 0) = \lsd(I - T_{\bar t + \varepsilon}, U, 0).
    \end{equation}
    
    Let us show that \eqref{eq:contradiction_assumption} cannot hold. Under the assumption that $\uo$ is the unique solution of $(I - T_t)(u) = 0$ in $\overline U$ for all $t \in [\bar t - \delta, \bar t + \delta]$, it holds that $\uto = (h_t^*)^{-1}(\uo)$ is the unique solution of $(I - P_t)(v) = 0$ in $\overline V_t = (h_t^*)^{-1}(\overline U)$, where $P_t$ is the operator defined in \eqref{eq:def_P_t}. Note that $P_t$ is differentiable at $\uto$ for every $t > 0$, with
    \begin{equation*}
        P_t'(\uto) = (-\Delta)^{-1}(f'(\uto)).
    \end{equation*}
    For $t \neq \bar t$ we have that $P_t'(\uto)$ is invertible, and then it can be shown (see \cite[Proposition 3.5.3]{Kesavan2004}) that
    \begin{equation}
        \label{eq:deg_P_deg_P_prime}
        \lsd(I - P_t, V_t, 0) = \lsd(I - P_t'(\uto), V_t, 0) = (- 1)^\ell,
    \end{equation}
    where $\ell$ is the number of eigenvalues of $P_t'(\uto)$ that are greater than $1$. We claim that $\ell = m(\uto)$, where $m(\uto)$ is the Morse index of $\uto$ in the space $H_0^1(\oto \cup \Gamma_t)$, which yields
    \begin{align}
        \lsd(I - P_t'(\uto), V_t, 0) = (-1)^{m(\uto)}. \label{eq:deg_P_prime}
    \end{align}
    Assuming for a moment that \eqref{eq:deg_P_prime} holds, by Theorem \ref{thm:computation_of_Morse_index} and since $0$ is a simple eigenvalue of $L_{u_{\bar t \omega}}$ in $H_0^1(\Omega_{\bar t \omega} \cup \Gamma_{\bar t})$
    \begin{equation*}
        m(u_{(\bar t + \varepsilon)\omega}) = m(u_{(\bar t - \varepsilon)\omega}) + 1
    \end{equation*}
    for all $\varepsilon > 0$ small enough. Hence
    \begin{equation*}
        \lsd(I - P_{\bar t + \varepsilon}'(u_{(\bar t + \varepsilon)\omega}), V_{\bar t + \varepsilon}, 0) = - \lsd(I - P_{\bar t - \varepsilon}'(u_{(\bar t - \varepsilon)\omega}), V_{\bar t - \varepsilon}, 0), 
    \end{equation*}
    which, in view of Lemma \ref{lem:deg_T_deg_P} and \eqref{eq:deg_P_deg_P_prime}, contradicts \eqref{eq:contradiction_assumption}. Therefore $(\bar t, \uo)$ is a bifurcation point, in the sense of Definition \ref{def:bifurcation}.
    
    \proofpart{2}{Proof of \eqref{eq:deg_P_prime}}
    
    Our aim is to show that for every eigenvalue of $P_t'(\uto)$ in the space $\holdsp(\oto \cup \Gamma_t)$ greater than one (which are known to be finitely many), there corresponds a negative eigenvalue of $L_{\uto}$, and vice-versa. This follows from the variational characterization of eigenvalues (see, e.g., \cite[Section 1.4]{DamascelliPacella2019}). 
    
    Let $\xi_k$ be an eigenfunction of $P_t(\uto)$ corresponding to the eigenvalue $\gamma_k$, beginning from the largest:
    \begin{equation*}
        \gamma_1 > \gamma_2 \geq \ldots > 0.
    \end{equation*}
    As $P_t'(\uto) \xi_k = \gamma_k \xi_k$, we have
    \begin{equation*}
        - \Delta \xi_k= \frac{1}{\gamma_k} f'(\uto) \xi_k
    \end{equation*}
    Multiplying by $\xi_k$ and integrating by parts we obtain
    \begin{equation*}
        \int_{\oto} |\nabla \xi_k|^2 \ dx = \frac{1}{\gamma_k} \int_{\oto} f'(\uto) \xi_k^2 \ dx.
    \end{equation*}
    Let us denote by $\mu_i$, $i \in \mathbb N^+$ the eigenvalues of $L_{\uto}$ in $H_0^1(\oto \cup \Gamma_t)$. Then, by the variational characterization of the $\mu_1$, we have:
    \begin{align}
        \mu_1
        & = \min_{\substack{\varphi \in H_0^1(\oto \cup \Gamma_t) \\ \varphi \neq 0}} \frac{\int_{\oto} |\nabla \varphi|^2 \ dx - \int_{\oto} f'(\uto) \varphi^2 \ dx}{\|\varphi\|_2^2} \nonumber \\
        & \leq \frac{\int_{\oto} |\nabla \xi_1|^2 \ dx - \int_{\oto} f'(\uto) \xi_1^2 \ dx}{\|\xi_1\|_2^2} \nonumber \\
        & = \frac{\left(\frac{1}{\gamma_1} - 1\right) \int_{\oto} f'(\uto) \xi_1^2 \ dx}{\|\xi_1\|_2^2} \nonumber \\
        & < 0 \nonumber.
    \end{align}
    Now, letting $\mathbf H_i$ denote an $i$-dimensional subspace of $H_0^1(\oto \cup \Gamma_t)$, we have:
    \begin{align}
        \mu_i 
        & = \min_{\mathbf H_i} \max_{\substack{\varphi \in \mathbf H_i \\ \varphi \neq 0}} \frac{\int_{\oto} |\nabla \varphi|^2 \ dx - \int_{\oto} f'(\uto) \varphi^2 \ dx}{\|\varphi\|_2^2} \nonumber \\
        & \leq \max_{\substack{\varphi \langle \xi_1, \ldots, \xi_i \rangle \\ \varphi \neq 0}} \frac{\int_{\oto} |\nabla \varphi|^2 \ dx - \int_{\oto} f'(\uto) \varphi^2 \ dx}{\|\varphi\|_2^2} \nonumber \\ 
        & = \frac{\int_{\oto} |\nabla \xi_i|^2 \ dx - \int_{\oto} f'(\uto) \xi_i^2 \ dx}{\|\xi_i\|_2^2} \nonumber \\
        & = \frac{\left(\frac{1}{\gamma_i} - 1\right) \int_{\oto} f'(\uto) \xi_1^2 \ dx}{\|\xi_i\|_2^2} \nonumber \\
        & < 0 \nonumber
    \end{align}
    for all $i > 1$ such that $\gamma_i > 1$. Hence we have
    \begin{equation*}
        m(\uto) \geq \# \{\gamma_i > 1\}.
    \end{equation*}
    On the other hand, if $\mu_i < 0$, we have:
    \begin{align*}
        0
        & > \mu_i \nonumber \\
        & = \max_{\mathbf H_{i - 1}} \min_{\substack{\varphi \perp \mathbf H_{i - 1} \\ \varphi \neq 0}} \frac{\int_{\oto} |\nabla \varphi|^2 \ dx - \int_{\oto} f'(\uto) \varphi^2 \ dx}{\|\varphi\|_2^2} \nonumber \\
        & \geq \min_{\substack{\varphi \perp \langle \xi_1, \ldots, \xi_{i - 1} \rangle \\ \varphi \neq 0}} \frac{\int_{\oto} |\nabla \varphi|^2 \ dx - \int_{\oto} f'(\uto) \varphi^2 \ dx}{\|\varphi\|_2^2} \nonumber \\
        & = \frac{\int_{\oto} |\nabla \xi_i|^2 \ dx - \int_{\oto} f'(\uto) \xi_i^2 \ dx}{\|\xi_i\|_2^2} \nonumber \\
        & = \frac{\left(\frac{1}{\gamma_i} - 1\right) \int_{\oto} f'(\uto) \xi_1^2 \ dx}{\|\xi_i\|_2^2}, \nonumber \\
    \end{align*}
    which immediately yields $\gamma_i > 1$. Hence $m(\uto) = \#\{\gamma_i > 1\}$ and \eqref{eq:deg_P_prime} holds.

    \item To show that the bifurcating solutions form a continuous branch near $\bar t$, the idea is to apply the inverse function theorem on appropriate subspaces after a Lyapunov-Schmidt reduction.

    Observe that the map $T_t$ is continuously differentiable, for all $t > 0$. Let us denote
    \begin{equation*}
        \mathcal X_1 \coloneqq \ker (I - T_{\bar t}'(\uo)),
    \end{equation*}
    and observe that
    \begin{equation*}
        \mathcal X_1 = h_{\bar t}^* \left(\ker (I - P_{\bar t}'(u_{\bar t \omega}) \right).
    \end{equation*}
    Since $0$ is a simple eigenvalue of $I - P_{\bar t}'(\uto)$ in $\holdsp(\Omega_{\bar t \omega} \cup \Gamma_{\bar t})$, then $\dim \mathcal X_1 = 1$. Let $w$ denote the corresponding normalized eigenfunction of $I - T_{\bar t}'(\uo)$, in such a way that
    \begin{equation*}
        \mathcal X_1 = \langle w \rangle = \{sw \ : \ s \in \mathbb R\}.
    \end{equation*}
    Then, setting
    \begin{equation*}
        \mathcal X_2 = \Range (I - T_{\bar t}'(\uo)),
    \end{equation*}
    we have
    \begin{equation*}
        \holdsp(\oo \cup \Gamma) = \mathcal X_1 \oplus \mathcal X_2.
    \end{equation*}

    Denoting by $P: \holdsp(\oo \cup \Gamma) \to \mathcal X_2$ the projection onto $\mathcal X_2$, then, solving $(I - T_t)(u) = 0$ is equivalent to solving, simultaneously, the two equations
    \begin{align}
        & P[(I - T_t)(u)] = 0 \label{eq:lyapunov_schmidt_1} \\
        & (I - P)[(I - T_t)(u)] = 0. \label{eq:lyapunov_schmidt_2}
    \end{align}
    
    Let us now consider the map $\Phi: \left((0, + \infty)\times \mathbb R\right) \times \mathcal X_2$ defined by
    \begin{equation*}
        \Phi((t, s), u) = P[(I - T_t)(\uo + sw + u)].
    \end{equation*}
    Let us observe that the map $\Phi$ is differentiable in the $u$ coordinate, and it holds
    \begin{align}
        \frac{\partial \Phi}{\partial u} ((\bar t, 0), 0)
        & = (P[(I - T_{\bar t}'(\uo))])|_{\mathcal X_2} \nonumber \\
        & = (I - T_{\bar t}'(\uo))|_{\mathcal X_2}, \nonumber
    \end{align}
    which is an isomorphism. Hence, by the implicit function theorem, we obtain the existence of $\varepsilon > 0$, a neighborhood $\mathcal U$ of $0$ in $\mathcal X_2$, and a $C^1$ map $\Psi: (\bar t - \varepsilon, \bar t + \varepsilon) \times (- \varepsilon, \varepsilon) \to \mathcal U$ such that
    \begin{equation*}
        P[(I - T_t)(\uo + sw + \Psi(t, s))] = 0 \quad \forall (t, s) \in (\bar t - \varepsilon, \bar t + \varepsilon) \times (- \varepsilon, \varepsilon)
    \end{equation*}
    and $\uo + sw + \Psi(t, s)$ are the unique solutions of the equation $P[(I - T_t)(u)] = 0$ in $\left((\bar t - \varepsilon, \bar t + \varepsilon) \times (-\varepsilon, \varepsilon)\right) \times \mathcal U$. 
    
    By \ref{it:local_bifurcation}, we know that the local bifurcation indeed occurs, so there exist solutions to \eqref{eq:lyapunov_schmidt_1}-\eqref{eq:lyapunov_schmidt_2} other than $u_\omega$, for all $t$ close enough to $\bar t$. Since the only possible such solutions are of the form $\uo + sw + \Psi(t, s)$, we conclude that these are precisely the bifurcating solutions.

    It remains to show that we can write $s = s(t)$, to have a curve of solutions. It suffices to show that $\frac{\partial \Psi}{\partial s}(\bar t, 0)$ is invertible and to apply the implicit function theorem again. Indeed, differentiating $(I - T_t)(\uo + sw + \Psi(t, s)) = 0$ with respect to $s$ yields
    \begin{equation*}
        (I - T_{\bar t}')\left(w + \frac{\partial \Psi}{\partial s}(\bar t, 0) \right) = 0,
    \end{equation*}
    which implies, since $w$ is a nontrivial eigenvector of $I - T_{\bar t}'$, that $\frac{\partial \Psi}{\partial s}(\bar t, 0) = - w$ and thus is invertible as a linear map. As anticipated, the implicit function theorem then yields the existence of $g: (-\varepsilon, \varepsilon) \to (\bar t - \varepsilon, \bar t + \varepsilon)$ such that the set of bifurcating solutions is given by
    \begin{equation*}
        \curve = \{\uo + g(t)w + \Psi(t, g(t)) \  : \ t \in (\bar t - \varepsilon, \bar t + \varepsilon)\}.
    \end{equation*}
    The claim of \ref{it:local_bifurcation} is proved.

    \item By the previous item, we have
    \begin{equation*}
        \|\uo - (\uo + g(t)w + \Psi(t, g(t)))\|_{C^{1, \alpha}} = \|g(t)w + \Psi(t, g(t))\|_{C^{1, \alpha}} \to 0 \quad \text{ as } \quad t \to \bar t,
    \end{equation*}
    which would not be possible if the bifurcating solutions had a different number of nodal regions.

    Indeed, let $0 < r_1 < \ldots < r_{n - 1} < r_n = 1$ be the roots of $\uo$ and denote $u_t \coloneqq \uo + g(t) + \Psi(t, g(t))$. Since we have $u_t \to \uo$ in $C^{1, \alpha}$ as $t \to 1$, then there exist $\varepsilon_0, \varepsilon > 0$ such that 
    \begin{equation}
    \label{eq:sign_a}
    \nabla u_t \cdot \nabla \uo > |\uo'(r_i)|^2 - \varepsilon > 0 \quad \text{ in } \omega \times [r_i - \varepsilon_0, r_i + \varepsilon_0], \quad i = 1, \ldots, n - 1,
    \end{equation}
    and also
    \begin{equation}
    \label{eq:sign_b}
        \nabla u_t \nabla \uo > |\uo'(1)| - \varepsilon > 0 \quad \text{ in } \omega \times [1 - \varepsilon_0, 1].
    \end{equation}
    
    On the other hand, outside of these regions, $u_t$ and $\uo$ have the same sign:
    \begin{equation}
    \label{eq:sign_c}
        u_t \uo > 0 \quad \text{ in } \oo \setminus \left(\bigcup_{i = 1}^{n - 1} (\omega \times [r_i - \varepsilon_0, r_i + \varepsilon_0]) \cup \omega \times [1 - \varepsilon_0, 1] \right).
    \end{equation}

    It is then clear that $u_t$ cannot have less nodal regions than $\uo$, for $u_t$ changes sign in $\omega \times [r_i - \varepsilon_0, r_i + \varepsilon_0]$, for $i = 1, \ldots, n - 1$. Suppose $u_t$ has more nodal regions than $\uo$. From \eqref{eq:sign_c}, it is clear that the additional nodal region is contained in some $\omega \times [r_i - \varepsilon_0, r_i + \varepsilon_0]$, for $i = 1, \ldots, n - 1$ (or in $\omega \times [1 - \varepsilon_0, 1]$. In this case, $u_t$ would have a critical point inside one of these regions. But this cannot happen, in view of \eqref{eq:sign_a}-\eqref{eq:sign_b}. 

    \item Let $\concomp$ be the connected component of $\mathcal S$ containing the point $(\bar t, \uo) \in (0, + \infty) \times \mathcal X(\oo \cup \Gamma)$. In particular, $\concomp$ contains the curve $\curve = \{\uo + g(t) w + \Psi(t, g(t)) \ : \ t \in (\bar t - \varepsilon, \bar t + \varepsilon)\}$ found in the proof of \ref{it:continuity}. Let us denote
    \begin{equation*}
        \holdsp_{x_N}(\oo \cup \Gamma) \coloneqq \{u \in \holdsp(\oo \cup \Gamma) \ : \ u(x', x_N) = u(x_N)\}.
    \end{equation*}
    
    For the sake of contradiction, let us assume that
    \begin{align}
        & \concomp \text{ is bounded }, \nonumber \\
        & t \geq \rho_0 > 0 \ \forall (t, u) \in \concomp, \label{eq:contradiction_global_bifurcation} \\
        &\concomp \cap \left((0, + \infty) \times \{\uo\}\right) = \{(\bar t, \uo)\}. \nonumber
    \end{align}

    Since $T_t$ is compact for every $t > 0$, then $\concomp$ is compact. Hence there exists a bounded open neighborhood $\mathcal A$ of $\concomp$ in $(0, + \infty) \times \holdsp(\oo \cup \Gamma)$ such that $\partial A \cap \mathcal S = \emptyset$ and
    \begin{equation*}
        \overline{\mathcal A} \cap (0, + \infty) = [\bar t - \delta, \bar t + \delta]
    \end{equation*}
    for some $\delta > 0$ small enough such that only $\bar t$ satisfies \eqref{eq:choice_bar_t} in this interval. The construction of such a neighborhood is standard. We omit it and refer to \cite[Section 29.1]{Deimling1985} or \cite[Lemma 4.6]{AmbrosettiMalchiodi2007} for the details.

    For $t \in [\bar t - \delta, \bar t + \delta]$, let us denote
    \begin{equation*}
        \mathcal A(t) \coloneqq \{u \in \holdsp(\oo \cup \Gamma) \ : \ (t, u) \in \mathcal A\}.
    \end{equation*}
    Since $\partial A \cap \sol = \emptyset$, the homotopy invariance of the Leray-Schauder degree implies that
    \begin{equation}
        \label{eq:constant_lsd}
        \lsd(I - T_t, \mathcal A(t), 0) \text{ is constant in } [\bar t - \delta, \bar t + \delta].
    \end{equation}
    Now, let $\mathcal B$ be a small open neighborhood of $\uo$ in $\holdsp(\oo \cup \Gamma)$, in such a way that $\mathcal B \subset \mathcal A(t)$ for all $t \in [\bar t - \delta_1, \bar t + \delta_1]$, for some $\delta_1 \in (0, \delta)$. Then the additivity property of the Leray-Schauder degree yields
    \begin{equation}
        \label{eq:split_lsd}
        \lsd(I - T_t, \mathcal A(t), 0) = \lsd(I - T_t, \mathcal A(t) \setminus \overline{\mathcal B}, 0) + \lsd(I - T_t, \mathcal B, 0), 
    \end{equation}
    for all $t \in [\bar t - \delta_1, \bar t + \delta_1]$. From the argument carried out for the proof of \ref{it:local_bifurcation}, we know that
    \begin{equation}
        \label{eq:change_lsd}
        \lsd(I - T_{\bar t + \delta_1}, \mathcal B, 0) = - \lsd(I - T_{\bar t - \delta_1}, \mathcal B, 0).
    \end{equation}
    On the other hand, up to taking $\mathcal B$ smaller, we can assume that there are no solutions of $(I - T_t)(u) = 0$ on $\partial (\mathcal A(t) \setminus \overline{\mathcal B})$ for $t < \bar t - \delta$ and $t > \bar t + \delta$. Let $\delta_2 > 0$ be such that $\mathcal A(\bar t + \delta_2) = \emptyset$. Then there are no solutions of $(I - T_t)(u) = 0$ on $\partial \left(\concomp \setminus ([\bar t + \delta, \bar t + \delta_2] \times \mathcal B)\right)$ and hence the homotopy invariance of the degree implies
    \begin{equation}
        \label{eq:no_lsd}
        \lsd(I - T_{\bar t + \delta}, \mathcal A(\bar t + \delta) \setminus \overline{\mathcal B}, 0) = \lsd(I - T_{\bar t + \delta_2}, \mathcal A(\bar t + \delta_2), 0) = 0.
    \end{equation}

    We see that \eqref{eq:split_lsd}-\eqref{eq:no_lsd} contradict \eqref{eq:constant_lsd}. Hence \eqref{eq:contradiction_global_bifurcation} cannot hold, and thus either $\mathcal K$ is unbounded, or $(0, u) \in \concomp$ for some $u \in \holdsp(\oo \cup \Gamma)$, or 
    \begin{equation}
        \label{eq:global_bifurcation_intersection}
        \mathcal P \coloneqq \mathcal K \cap \left((0, + \infty) \times \{\uo\}\right) \neq \emptyset.
    \end{equation}
    \end{enumerate}
    The proof is complete.
\end{proof}

\begin{cor}
    \label{cor:existence}
    Let $\omega \subset \mathbb R^{N - 1}$ be a smooth bounded domain and let $f \in C^{1, \alpha}(\mathbb R)$ be a nonlinearity satisfying \ref{it:f1}-\ref{it:f3}, and let $\uo$ be a one-dimensional solution of \eqref{eq:pde} with $n(\uo) = n$ nodal domains. Let $\bar t > 0$ such that $u_{\bar t\omega}$ is degenerate in the space $H_0^1(\Omega_{\bar t\omega} \cup \Gamma_{\bar t})$. Then, if $0$ is a simple eigenvalue of $L_{u_{\bar t \omega}}$ in $H_0^1(\Omega_{\bar t \omega} \cup \Gamma_{\bar t})$, for $t$ close to $\bar t$, there exists a solution to
    \begin{equation*}
        \left\{
        \begin{array}{rcll}
        - \Delta u & = & f(u) & \quad \text{ in } \oto \\
        u & = & 0 & \quad \text{ on } \Gamma_{0, t} \\
        \displaystyle \frac{\partial u}{\partial \nu} & = & 0 & \quad \text{ on } \Gamma_t
        \end{array}
        \right.
        ,
    \end{equation*}
    that is not one-dimensional and has $n$ nodal domains. In particular, if $\uo$ is positive, then there is a positive solution that is not one-dimensional.
\end{cor}
\begin{proof}
    It suffices to take the family of bifurcating solutions given by Theorem \ref{thm:bifurcation} and transport them back to $\oto$ via $(h_t^*)^{-1}$, taking into account Remark \ref{rem:symmetry_breaking}.
\end{proof}

\begin{remark}
    Observe that the assumption on the simplicity of $0$ as a simple eigenvalue for $L_{u_{\bar t \omega }}$ is actually an assumption on the interplay between the eigenvalues of \eqref{eq:auxiliary_one_dim_eigenvalue_problem} and the geometry of $\omega$. Indeed, this assumption means that there exists a unique pair $(i, j) \in \mathbb N^+ \times \mathbb N$ such that if we scale $\omega$ by a factor of $\bar t$ we obtain 
    \begin{equation}
        \label{eq:choice_bar_t}
        \frac{1}{\bar t^2} \lambda_j = - \alpha_i.
    \end{equation} 
   It is clear that if there exist multiple pairs $(i, j)$ satisfying \eqref{eq:choice_bar_t} (for different $\bar t$, of course), then, repeating the arguments carried out in the proof of Theorem \ref{thm:bifurcation}, we get multiple bifurcation.
\end{remark}

Let us now illustrate cases where the assumptions on the multiplicity of $0$ as an eigenvalue of $L_{u_{\bar t\omega}}$ in $H_0^1(\Omega_{\bar t \omega} \cup \Gamma_{\bar t})$ hold.

Let $\uo$ be a positive one-dimensional solution and assume that there exists a simple Neumann eigenvalue $\bar \lambda$ for $- \Delta_{\plane}$ in $\omega$. Then we can take $\bar t$ such that
\begin{equation*}
    \frac{1}{\bar t^2} \bar \lambda = - \alpha_1,
\end{equation*}
where $\alpha_1$ is the unique negative eigenvalue of \eqref{eq:auxiliary_one_dim_eigenvalue_problem}. Since the eigenvalues of the Neumann-Laplacian are generically all simple for smooth bounded domains in the Euclidean space (see \cite[Example 6.4]{Henry2005}), the simplicity assumption holds for ``almost every" shape $\omega$.

Some interesting shapes, however, are known to possess eigenvalues of multiplicity greater than one, e.g., if $\omega$ is a ball in $\plane$ or has some other symmetry. Even in these cases, we still get the existence of solutions of \eqref{eq:pde} that are not one-dimensional by slightly modifying our arguments. In fact, it suffices to work in some space of functions that are invariant with respect to some symmetry group action. For example, let $\omega$ be a ball in $\plane$ and consider the subspace
\begin{equation*}
\mathcal E \coloneqq \{v \in \holdsp(\oo \cup \Gamma) \ : \ v(x', x_N) = v(T(x'), x_N) \ \forall \ T \in SO(N - 1)\}
\end{equation*}
of functions in $\oo$ that are invariant with respect to rotations of the base $\omega$. In this case, the Neumann eigenvalues of $- \Delta_{\plane}$ in $\omega$ are simple in the space $\mathcal E$. For more details on the matter of multiple Neumann eigenvalues, we refer to \cite{MarrocosPereira2015}.

\begin{remark} 
Observe that the proof of \ref{it:local_bifurcation} of Theorem \ref{thm:bifurcation} still works under the slightly weaker assumption that
$0$ has odd multiplicity as an eigenvalue of $L_{u_{\bar t \omega}}$, because even in this case we would have the change of Leray-Schauder degree for the one-dimensional solution $\uto$ when $t$ crosses $\bar t$. 
\end{remark}


\section*{Acknowledgements}
This work has been funded by the European Union - NextGenerationEU within the framework of PNRR  Mission 4 - Component 2 - Investment 1.1 under the Italian Ministry of University and Research (MUR) program PRIN 2022 - grant number 2022BCFHN2 - Advanced theoretical aspects in PDEs and their applications - CUP: H53D23001960006.

I express my gratitude to Filomena Pacella and Francesca Gladiali for many insightful discussions.

\bibliographystyle{acm}
\bibliography{ref_math}

\begin{thebibliography}{10}

\bibitem{AfonsoIacopettiPacella2023Cheeger}
{\sc Afonso, D.~G., Iacopetti, A., and Pacella, F.}
\newblock Overdetermined problems and relative {C}heeger sets in unbounded
  domains.
\newblock {\em Atti Accad. Naz. Lincei Cl. Sci. Fis. Mat. Natur 34}, 2 (2023),
  531--546.

\bibitem{AfonsoIacopettiPacella2024Energypublished}
{\sc Afonso, D.~G., Iacopetti, A., and Pacella, F.}
\newblock Energy stability for a class of semilinear elliptic problems.
\newblock {\em J. Geom. Anal. 34}, 75 (2024).

\bibitem{AmbrosettiMalchiodi2007}
{\sc Ambrosetti, A., and Malchiodi, A.}
\newblock {\em Nonlinear Analysis and Semilinear Elliptic Problems}.
\newblock Cambridge University Press, 2007.

\bibitem{BadialeSerra2011}
{\sc Badiale, M., and Serra, E.}
\newblock {\em Semilinear Elliptic Equations for Beginners}.
\newblock Springer, 2011.

\bibitem{BerestyckiCaffarelliNirenberg1997b}
{\sc Berestycki, H., Caffarelli, L., and Nirenberg, L.}
\newblock Further qualitative properties for elliptic equations in unbounded
  domains.
\newblock {\em Annali della Scuola Normale Superiore di Pisa - Classe di
  Scienze Ser. 4, 25}, 1-2 (1997), 69--94.

\bibitem{BerestyckiCaffarelliNirenberg1997a}
{\sc Berestycki, H., Caffarelli, L., and Nirenberg, L.}
\newblock Monotonicity for elliptic equations in unbounded {L}ipschitz domains.
\newblock {\em Commun. Pure Appl. Math.\/} (1997).

\bibitem{BerestyckiNirenberg1988}
{\sc Berestycki, H., and Nirenberg, L.}
\newblock Monotonicity, symmetry and antisymmetry of solutions of semilinear
  elliptic equations.
\newblock {\em J. Geom. Phys. 5}, 2 (1988), 237--275.

\bibitem{BerestyckiNirenberg1990}
{\sc Berestycki, H., and Nirenberg, L.}
\newblock {\em Analysis et cetera}.
\newblock Academic Press, 1990, ch.~Some qualitative properties of solutions of
  semilinear elliptic equations in cylindrical domains, pp.~115--164.

\bibitem{BerestyckiNirenberg1991}
{\sc Berestycki, H., and Nirenberg, L.}
\newblock On the {M}ethod of {M}oving {P}lanes and the {S}liding {M}ethod.
\newblock {\em Bol. Soc. Bras. Mat 22}, 1 (1991), 1--37.

\bibitem{ChenWuYao2023}
{\sc Chen, H., Wu, K., and Yao, R.}
\newblock Qualitative properties of nonnegative solutions of some semilinear
  elliptic equations in cylindrical domains.
\newblock {\em Calc. Var. Partial Differential Equations 62}, 6 (2023).

\bibitem{CiraoloPacellaPolvara2023}
{\sc Ciraolo, G., Pacella, F., and Polvara, C.}
\newblock Symmetry breaking and instability for semilinear elliptic equations
  in spherical sectors and cones.
\newblock {\em arXiv:2305.10176v1\/} (2023).

\bibitem{Cortazaretal2016}
{\sc Cortázar, C., Elgueta, M., and García-Melián, J.}
\newblock Nonnegative solutions of semilinear elliptic equations in
  half-spaces.
\newblock {\em J. Math. Pures Appl. (9) 106\/} (2016), 866--876.

\bibitem{DamascelliPacella2019}
{\sc Damascelli, L., and Pacella, F.}
\newblock {\em Morse Index of Solutions of Nonlinear Elliptic Equations}.
\newblock De Gruyter, 2019.

\bibitem{Deimling1985}
{\sc Deimling, K.}
\newblock {\em Nonlinear Functional Analysis}.
\newblock Springer, 1985.

\bibitem{delPinoKowalczykWei2011}
{\sc del Pino, M., Kowalczyk, M., and Wei, J.}
\newblock On {D}e {G}iorgi's conjecture in dimension $n \geq 9$.
\newblock {\em Ann. of Math. (2) 174\/} (2011).

\bibitem{FallMinlendWeth2017}
{\sc Fall, M.~M., Minlend, I.~A., and Weth, T.}
\newblock Unbounded periodic solutions to {S}errin's overdetermined boundary
  value problem.
\newblock {\em Arch. Ration. Mech. Anal. 223}, 2 (2017), 737--759.

\bibitem{FarinaSciunzi2017}
{\sc Farina, A., and Sciunzi, B.}
\newblock Monotonicity and symmetry of nonnegative solutions to $-{\Delta
  u=f(u)}$ in half-planes and strips.
\newblock {\em Advanced Nonlinear Studies 17}, 2 (2017), 297--310.

\bibitem{GidasNiNirenberg1979}
{\sc Gidas, B., Ni, W.-M., and Nirenberg, L.}
\newblock Symmetry and related properties via the maximum {P}rinciple.
\newblock {\em Comm. Math. Phys. 68\/} (1979), 209--243.

\bibitem{GidasNiNirenberg1981}
{\sc Gidas, B., Ni, W.-M., and Nirenberg, L.}
\newblock Symmetry of positive solutions of nonlinear elliptic equations in
  $\mathbb{R}^n$.
\newblock In {\em Mathematical Analysis and Applications, Part A\/} (1981),
  vol.~7A of {\em Avances in Mathematics Supplementary Studies}, pp.~369--402.

\bibitem{Gladiali2010}
{\sc Gladiali, F.}
\newblock A global bifurcation result for a semilinear elliptic equation.
\newblock {\em J. Math. Anal. Appl. 369}, 1 (2010), 306--311.

\bibitem{Gladiali2017}
{\sc Gladiali, F.}
\newblock Separation of branches of ${O}({N} - 1)$-invariant solutions for a
  semilinear elliptic equation.
\newblock {\em J. Math. Anal. Appl. 453\/} (2017), 159--173.

\bibitem{GladialiGrossiPacellaSrikanth2010}
{\sc Gladiali, F., Grossi, M., Pacella, F., and Srikanth, P.~N.}
\newblock Bifurcation and symmetry breaking for a class of semilinear elliptic
  equations in an annulus.
\newblock {\em Calc. Var. Partial Differential Equations 40}, 3-4 (jun 2010),
  295--317.

\bibitem{GladialiIanni2020}
{\sc Gladiali, F., and Ianni, I.}
\newblock Quasi-radial solutions for the {L}ane-{E}mden problem in the ball.
\newblock {\em Nonlinear Differ. Equ. Appl.\/} (2020).

\bibitem{Hartman2002book}
{\sc Hartman, P.}
\newblock {\em Ordinary Differential Equations}, 2nd~ed.
\newblock No.~38 in SIAM's Classics in Applied Mathematics. SIAM, 2002.

\bibitem{Henry2005}
{\sc Henry, D.}
\newblock {\em Perturbation of the Boundary in Boundary-Value Problems of
  Partial Differential Equations}.
\newblock Cambridge University Press, may 2005.

\bibitem{Kesavan2004}
{\sc Kesavan, S.}
\newblock {\em Nonlinear Functional Analysis}.
\newblock Hindustan Book Agency, 2004.

\bibitem{MarrocosPereira2015}
{\sc Marrocos, M., and Pereira, A.~L.}
\newblock Eigenvalues of the {N}eumann-{L}aplacian in symmetric regions.
\newblock {\em J. Math. Phys. 56\/} (2015).

\bibitem{NiNussbaum1985}
{\sc Ni, W.-M., and Nussbaum, R.~D.}
\newblock Uniqueness and nonuniqueness for positive radial solutions of
  ${\Delta u} + f(u, r) = 0$.
\newblock {\em Commun. Pure Appl. Math. 38\/} (1985), 67--108.

\bibitem{Nussbaum1971}
{\sc Nussbaum, R.~D.}
\newblock The fixed point index for local condensing maps.
\newblock {\em Ann. Mat. Pura Appl. (4) 89}, 1 (dec 1971), 217--258.

\bibitem{PacellaRuizSicbaldi2024}
{\sc Pacella, F., Ruiz, D., and Sicbaldi, P.}
\newblock Nontrivial solutions to the relative overdetermined torsion problem
  in a cylinder.
\newblock {\em arXiv:2404.09272\/} (2024).

\bibitem{PacellaTralli2020}
{\sc Pacella, F., and Tralli, G.}
\newblock Overdetermined problems and constant mean curvature surfaces in
  cones.
\newblock {\em Rev. Mat. Iberoam. 36\/} (2020), 841--867.

\bibitem{Rabinowitz1971}
{\sc Rabinowitz, P.~H.}
\newblock Some global results for nonlinear eigenvalue problems.
\newblock {\em J. Funct. Anal. 7\/} (1971), 487--513.

\bibitem{Rabinowitz1986}
{\sc Rabinowitz, P.~H.}
\newblock {\em Minimax Methods in Critical Point Theory with Applications to
  Differential Equations}.
\newblock American Mathematical Society, 1986.

\bibitem{RosRuizSicbaldi2017}
{\sc Ros, A., Ruiz, D., and Sicbaldi, P.}
\newblock A rigidity result for overdetermined elliptic problems in the plane.
\newblock {\em Commun. Pure Appl. Math. 70}, 7 (apr 2017), 1223--1252.

\bibitem{RosRuizSicbaldi2019}
{\sc Ros, A., Ruiz, D., and Sicbaldi, P.}
\newblock Solutions to overdetermined elliptic problems in nontrivial exterior
  domains.
\newblock {\em Journal of the European Mathematical Society 22}, 1 (sep 2019),
  253--281.

\bibitem{Serrin1971}
{\sc Serrin, J.}
\newblock A {S}ymmetry {P}roblem in {P}otential {T}heory.
\newblock {\em Arch. Ration. Mech. Anal. 43}, 4 (1971), 304--318.

\bibitem{Sicbaldi2010}
{\sc Sicbaldi, P.}
\newblock New examples of extremal domains for the first eigenvalue of the
  {L}aplacian in flat tori.
\newblock {\em Calc. Var. Partial Differential Equations 37\/} (2010),
  329--344.

\bibitem{Willem1996}
{\sc Willem, M.}
\newblock {\em Minimax Theorems}.
\newblock Birkhäuser, 1996.

\end{thebibliography}

\end{document}